\documentclass[12pt]{amsart}
\usepackage{amsmath, amsthm, color}
\usepackage{graphicx}
\usepackage{caption}
\usepackage{mathtools}
\usepackage{enumitem}
\usepackage{xfrac}
\usepackage{verbatim}
\usepackage{tikz}
\usetikzlibrary{matrix}
\usetikzlibrary{arrows}
\usepackage{algorithm}
\usepackage[noend]{algpseudocode}
\usepackage{caption}
\usepackage{subcaption}

\usepackage{makecell}
\usepackage{array}

\newtheorem{theorem}{Theorem}
\newtheorem{proposition}[theorem]{Proposition}
\newtheorem{lemma}[theorem]{Lemma}
\newtheorem{corollary}[theorem]{Corollary}

\theoremstyle{definition}

\newtheorem{remark}[theorem]{Remark}
\newtheorem{conjecture}[theorem]{Conjecture}

\newtheorem{example}[theorem]{Example}

\newcommand{\RR}{\mathbb{R}}

\newcommand{\CC}{\mathbb{C}}

\newcommand{\Loss}{\mathcal{L}}

\title{Complex Critical Points of Deep Linear Neural Networks}
\author{Ayush Bharadwaj and Serkan Ho\c{s}ten}
\date{}

\begin{document}


\maketitle

\begin{abstract}
We extend the work of Mehta, Chen, Tang, and Hauenstein on computing the complex critical points of the loss function of deep linear neutral networks when the activation function is the identity function. For networks with a single hidden layer trained on a single data point we give an improved bound on the number of complex critical points of the loss function. We show that for any number of hidden layers complex critical points with zero coordinates arise in certain patterns which we completely classify for networks with one hidden layer. We report our results of computational experiments with varying network architectures defining 
small deep linear networks using {\em HomotopyContinuation.jl}.

\end{abstract}


\section{Introduction} \label{intro}

Machine learning applications through deep learning techniques have been enormously successful in areas ranging from natural language processing and object recognition to drug discovery \cite{BGC15, LBH15}. These techniques involve optimizing a non-convex loss (cost) function such as the mean squared error between the observed and predicted
data from the underlying neural network. Typically, this is an NP-hard problem \cite{BR88} where the loss function has numerous local minima. Though methods such as stochastic gradient descent are empirically observed to close in to 
{\em good enough} local minima (see \cite{BH89,GT92,YC95}), a theoretical understanding for the success of deep learning is incomplete. 

A deep {\em linear} neural network is a simpler model on which detailed analysis is possible \cite{BH89}. They differ from more general deep neural networks in the activation function used: linear neural networks employ linear activation functions instead of typical nonlinear ones such as ReLU. Nevertheless, the loss function of the deep linear networks is non-convex, and these networks exhibit many characteristics found in deep nonlinear networks. They have been used as a testing ground for ideas in general artificial neural networks \cite{BH95}. 

Our starting point is the work of Mehta, Chen, Tang, and Hauenstein \cite{MCTH22} in which the authors employ techniques from (numerical) algebraic geometry to compute and analyze complex and real critical points of the loss function of deep linear networks. They introduce a generalized Tikhonov regularization of the loss function (see below) and demonstrate that this turns positive dimensional components of the set of critical points of the loss function (known as flat critical sets) into isolated nondegenerate critical points of its regularized version. The resulting system of polynomial equations are sparse, and the Bernshtein-Kushnirenko-Khovanskii (BKK) bound \cite{BKK, HS95} gives an upper bound on the number of 
the complex critical points of the regularized loss function. Further, Mehta et. al. use numerical homotopy continuation methods \cite{TYLi03, SW05} implemented in 
{\it Bertini} \cite{BHSW} to find the critical points
of deep linear neural networks with one hidden layer trained on up to five data points. They report that, in these instances, the numbers of complex and real critical points are substantially lower than the corresponding BKK bound and the Dedieu-Malajovich bound \cite{DM08} for the average number of real solutions for systems with normally distributed coefficients.
 
\subsection*{Deep linear neural networks}
A feedforward neural network is a function 
$f \, : \, \RR^{d_x} \longrightarrow \RR^{d_y}$ with
$$f  = \omega_{H+1} \circ \phi  \circ \omega_H \cdots \phi \circ \omega_2 \circ \phi \circ \omega_1$$
where $\omega_i \, : \,  \RR^{d_{i-1}} \longrightarrow \RR^{d_i}$ is an affine transformation for $i=1,\ldots, H+1$
and $\phi \, : \, \RR^{d_*} \longrightarrow \RR^{d_*}$
is the activation function. The activation function is applied to each coordinate of a vector and is a non-linear 
function in most applications. In a fixed basis $\omega_i$
is represented by a $d_i \times d_{i-1}$ weight matrix $W_i$
together with a bias vector $b_i \in \RR^{d_i}$. Note that
$W_1$ is a $d_1 \times d_x$ matrix and $W_{H+1}$ is a $d_y \times d_H$ matrix.  The weight matrices $W_1, \ldots, W_H$
correspond to the hidden layers of artificial neurons whereas $W_{H+1}$ corresponds to the output layer of the network. 

In a deep {\em linear} neural network the activation function is chosen to be an affine function. Hence in this 
case we can assume $\phi$ to be the identity function. We will further assume that all bias vectors are zero vectors, i.e. all $\omega_i$ are linear transformations determined
by the entries (weights) in $W_i$, $i=1,\ldots, H+1$. We let
the number of weights to be $N = d_xd_1 + d_1d_2 + \cdots + d_Hd_y$.

\subsection*{The loss function and its regularization}

To train a deep linear neural network, $m$ data points $x_1, \ldots, x_m$ in $\RR^{d_x}$ are paired with $m$ output vectors $y_1, \ldots, y_m \in \RR^{d_y}$. We collect the input and output vectors as columns of two matrices $X \in \RR^{d_x \times m}$ and $Y \in \RR^{d_y \times m}$. Typically, one needs to determine the weight matrices $W_1, \ldots, W_{H+1}$ that will minimize the loss function:
$$ \Loss(W) = \frac{1}{2} \sum_{i=1}^m \lVert W_{H+1}W_H \cdots W_1x_i - y_i \rVert_2^2.$$
We will follow \cite{MCTH22} to regularize the loss function by using regularization matrices $\Lambda_1, 
\ldots, \Lambda_{H+1}$ where each $\Lambda_i$ is the same size as $W_i$. We arrive at the regularized loss function
\begin{equation}
\label{eq:regLoss}
\Loss^\Lambda(W) = \Loss(W) + \frac{1}{2} \left( \lVert \Lambda_1 \circ W_1 \rVert_F^2 + \cdots + \lVert \Lambda_{H+1} \circ W_{H+1} \rVert_F^2 \right)
\end{equation}
where $\Lambda_i \circ W_i$ denotes the entrywise (Hadamard) product of $\Lambda_i$ and $W_i$. This choice of regularization is justified by the results (see Theorem 1 
and Theorem 2 in \cite{MCTH22}) that for almost all choices of the regularization matrices the solutions to 
the gradient system $\nabla \Loss^\Lambda = 0$ are isolated
and nondegenerate. 

\subsection*{Our contributions}
This paper is about the complex critical points of $\Loss^\Lambda$, namely, the solutions to $\nabla \Loss^\Lambda = 0$. This gradient system is a system of $N$
polynomial equations in $N$ variables (weights) where 
each polynomial has degree $2H+1$. This set up
opens the door to using (numerical) algebraic geometry to bound the number of critical points as well as to compute them. 

In Section~\ref{sec2} we will summarize the gradient equations for the regularized loss function  of deep linear
neural networks. In Section~\ref{sec3} we will prove a bound on the number of complex critical points of the loss function of a linear neural network with one hidden layer trained on a single data point. We achieve this bound by 
rewriting the gradient equations. Our new bound is much better than the classic B\'ezout bound as well as the BKK bound previously observed \cite[Table I]{MCTH22} and confirmed and expanded by us (see Table \ref{tab:Table_H1m1} in the Appendix). 

Section~\ref{sec4} is devoted to understanding critical points where certain weights are zero. We show that, for any deep linear network trained
on a single data point, if a weight is zero, then 
the entire row or column containing that weight must be zero. We also prove that row $k$  of the weight
matrix $W_i$ is entirely zero if and only if the column $k$ of the weight matrix $W_{i+1}$ is entirely zero. As a corollary, we completely classify the zero patterns of complex critical points for linear networks with one hidden
layer trained on a single data point. 

In Section~\ref{sec5} we turn to our computational experiments. Using {\it HomotopyContinuation.jl} \cite{HomotopyContinuation} we solve the gradient equations 
for different parameter values of $d_x, d_y$ and $d_i$ when
$H, m \in \{1,2\}$. These extend the computational horizon beyond the one set in \cite{MCTH22}. We believe the extended data we provide will inspire further results.

\section{Gradient Equations for the Loss Function}
\label{sec2}
The critical points of the regularized loss function $\Loss^\Lambda(W)$ are the solutions to $\nabla \Loss^\Lambda = 0$, the gradient equations. We derive the gradient equations in a compact form. 

Let $W=W_{H+1}W_H\cdots W_1$, and let $U_i^T = \prod_{j=i+1}^{H+1} W_j^T$ and $V_i^T = \prod_{j=1}^{i-1} W_j^T$. Then 
\begin{equation}
\label{eq:gradLoss}
\frac{\partial \Loss^\Lambda}{\partial W_i} = U_i^T\left(W \left( \sum_{k=1}^m x_kx_k^T \right) - \left(\sum_{k=1}^m y_kx_k^T \right) \right)V_i^T + \Lambda_i \circ W_i
\end{equation}
where the $(j,k)$ entry of the matrix $\frac{\partial \Loss^\Lambda}{\partial W_i}$ is the partial derivative of 
$\Loss^\Lambda$ with respect to the $(j,k)$ entry of $W_i$.

It is easy to observe that the resulting system consists of $N$ polynomials in $N$ variables. The constant term in each polynomial is zero, therefore $0 \in \CC^N$ is always a critical point.  The terms of highest degree in every polynomial have degree $2H+1$. This leads to an explicit classical B\'{e}zout bound (CBB).
\begin{proposition}\cite[Proposition 3]{MCTH22} \label{prop:Bezout}
The regularized loss function $\Loss^\Lambda$ has at most
$(2H+1)^N$ complex critical points.
\end{proposition}

\begin{example} \label{ex:4variables}
Let us consider a 2-layer network with $W_1 = [\alpha_1 \ \alpha_2]$, $W_2 = \begin{bsmallmatrix} \beta_1\\ \beta_2 \end{bsmallmatrix}$, $X = \begin{bsmallmatrix} 1 & 2\\ 3 & 4 \end{bsmallmatrix}$, $Y = \begin{bsmallmatrix} 1 & 3\\ 2 & 4 \end{bsmallmatrix}$, $\Lambda_1 = \begin{bsmallmatrix} 4 & -3 \end{bsmallmatrix}$, and $\Lambda_2 = \begin{bsmallmatrix} -2 \\ 5 \end{bsmallmatrix}$. 
Using these in (\ref{eq:gradLoss}), we get the gradient system
\begin{equation*}
    \begin{split}
    5\alpha_1\beta_1^2 + 5\alpha_1\beta_2^2 + 11\alpha_2\beta_1^2 + 11\alpha_2\beta_2^2 -7\beta_1 -10 \beta_2 +4\alpha_1&= 0\\
    11\alpha_1 \beta_1^2 + 11\alpha_1\beta_2^2 + 25\alpha_2\beta_1^2 + 25\alpha_2 \beta_2^2 - 15\beta_1 - 22\beta_2 -3 \alpha_2&= 0\\
    5\alpha_1^2\beta_1 + 22\alpha_1\alpha_2\beta_1 + 25 \alpha_2^2\beta_1 -7\alpha_1 -15\alpha_2 -2 \beta_1&= 0\\
    5 \alpha_1^2\beta_2 + 22\alpha_1\alpha_2\beta_2 + 25\alpha_2^2\beta_2 -10\alpha_1 -22\alpha_2 +5\beta_2&= 0
    \end{split}
\end{equation*}
\hfill$ \triangle$
\end{example}
As the above example indicates the gradient equations are sparse. This allows the use of the BKK bound for the number of critical points in $(\CC^*)^N$ as well as a modified
BKK bound for the number of critical points in $\CC^N$ 
\cite{LW96, RW96}. For instance, the BKK bound for the solutions in $\CC^4$ for the gradient equations in Example~\ref{ex:4variables} is $33$. The actual number of solutions in $(\CC^*)^4$ and $\CC^4$ are $16$ and $17$, respectively. The B\'{e}zout
bound gives $3^4=81$. Our computational experiments in 
Section~\ref{sec5} (also see \cite{MCTH22}) will show that generally all of these bounds are far from the actual number of critical points.  Starting from the next section
we take a closer look into improving the BKK bound. We close this section with the following result.

\begin{proposition} \label{prop:sameBKK}
Each polynomial in the gradient equations (\ref{eq:gradLoss}) has the same monomial support for generic input and output vectors as $m$, the number of these vectors, varies while 
the rest of the parameters, namely, $H$, $d_x$, $d_y$, and
 $d_i$, $i=1, \ldots, H$ stay constant. In particular, BKK bounds on the number of complex critical points for these systems in $(\CC^*)^N$ and $(\CC)^N$  are the same bounds
 for any $m$.   
\end{proposition}
\begin{proof} First, we replace in (\ref{eq:gradLoss})
the matrix $\sum_{k=1}^m x_kx_k^T$ with the $d_x \times d_x$
symmetric matrix of indeterminates $Z = (z_{ij})$, and the 
matrix $\sum_{k=1}^m y_kx_k^T$ with the $d_y \times d_x$ matrix of indeterminates $T = (t_{ij})$. Multiplying out
the matrices in (\ref{eq:gradLoss}) given polynomials in the entries of $W_1, \ldots, W_{H+1}$ with coefficients in $z_{ij}$ and $t_{ij}$. These polynomials have terms only in degrees $2H+1$, $H$, and $1$. The terms of degree $2H+1$ that actually appear in a polynomial have coefficients that are linear in the entries of $Z$, and the terms of degree $H$ that appear
have coefficients that are linear in the entries of $T$. 
Finally, the single monomial of degree one that appears in a polynomial has an entry of some $\Lambda_i$
as a coefficient. We claim that for generic $x_1, \ldots, x_m$ and $y_1, \ldots, y_m$ the linear coefficient polynomials in $Z$ or $T$ do not vanish. This follows from the fact that determinantal varieties of $d_x \times d_x$ symmetric matrices and determinantal varieties of $d_y\times d_x$ matrices of rank $\leq r$ are not contained in linear subspaces. Genericity of $\Lambda_i$ guarantees that the 
required monomial of degree one is also present. 
\end{proof}

\section{Networks with one hidden layer trained on one data point}	
\label{sec3}
This section provides evidence that, by carefully looking at the gradient equations, one can gain insight into the number of critical points of the regularized loss function. 
We will treat linear neural networks with one hidden layer trained on a single data point ($H=1$ and $m=1$) with no restrictions on the dimensions of the input and output vectors ($d_x =n $ and $d_y = p$) as well as on the number of neurons in the hidden layer ($d_1 = d$). Although this may be viewed unrealistic for real world applications it is meant to be a point of departure for more work from symbolic and numerical algebraic geometry. 

To simplify the exposition we let $A=W_1$ and $B=W_2$. 
We also use $\Lambda= \Lambda_1$ and $\Sigma = \Lambda_2$.
If we declare $S_i = \sum_{i=1}^n a_{ij}x_j$ for $i=1, \ldots, d$ and $R_k = \sum_{i=1}^d b_{ki}S_i - y_k$ for $k=1, \ldots, p$, then 
\begin{equation*}\label{e_I}
    Wxx^T - yx^T = 
    \begin{pmatrix}
         R_1x_1 & \cdots & R_1 x_n\\
          \vdots & & \vdots \\
         R_p x_1 & \cdots & R_p x_n\\
    \end{pmatrix}.
\end{equation*}
Now $B^T(Wxx^T - yx^T) + \Lambda \circ A = 0$ yields
\begin{equation}\label{e_a}
    \begin{split}
        \left(\sum_{k=1}^p b_{k1}R_k \right) x_j &= -\lambda_{1j}a_{1j}\ \ \ \ \ \ j = 1, \ldots, n\\
        \vdots&\\
        \left( \sum_{k=1}^p b_{kd}R_k \right) x_j &= -\lambda_{dj}a_{dj}\ \ \ \ \ \ j = 1, \ldots, n
    \end{split}
\end{equation}
Similarly, $(Wxx^T - yx^T)A^T + \Sigma \circ B = 0$ gives rise to 
\begin{equation}\label{e_b}
   R_k S_i = -\sigma_{ki}b_{ki} \ \ \ \ i=1, \ldots, d \ \ \ \ k=1, \ldots, p. 
\end{equation}
\begin{lemma} For generic input and output vector $x$ and $y$ and generic regularization matrices $\Lambda$ and $\Sigma$ we obtain
\begin{equation}\label{e_a'}
    \begin{split}
        a_{1j} &= \left(
        \frac{x_j}{\lambda_{1j}}\frac{\lambda_{11}}{x_1}
        \right) a_{11}\ \ \ \ \ \ j = 1, \ldots, n\\
        &\ \vdots\\
        a_{dj} &= \left(
        \frac{x_j}{\lambda_{dj}}\frac{\lambda_{d1}}{x_1}
        \right) a_{d1}\ \ \ \ \ \ j = 1, \ldots, n\\
    \end{split}
\end{equation}
and 
\begin{equation}\label{e_A'}
        S_i = \left( \sum_{j=1}^n \frac{x_j^2}{\lambda_{ij}}\right)\frac{\lambda_{i1}}{x_1}a_{i1}\ \ \ \ i = 1, \ldots, d
\end{equation}
\end{lemma}
\begin{proof}
For generic data, the equation $B^T(Wxx^T - yx^T) + \Lambda \circ A = 0$ implies that $\Lambda \circ A$ is a matrix of rank at most one since $xx^T$ and $yx^T$ are matrices of rank one. The identities in (\ref{e_a'}) are an explicit version of this observation. They follow from the fact that $\left(\sum_{k=1}^p b_{k1}R_k \right)$, i.e. the coefficient of $x_j$ in the first row of (\ref{e_a}), is a constant complex number that does not depend on $j$. Therefore $\frac{-\lambda_{1j}a_{1j}}{x_j}$ is a constant. Hence, we have $\frac{-\lambda_{11}a_{11}}{x_1} = \frac{-\lambda_{12}a_{12}}{x_2} = \cdots = \frac{-\lambda_{1n}a_{1n}}{x_n}$. Equivalently, $a_{1j} = \left( \frac{x_j}{\lambda_{1j}}\frac{\lambda_{11}}{x_1} \right) a_{11}$ for all $j = 1, \ldots, n$. Repeating the same for all equations in (\ref{e_a}) gives (\ref{e_a'}). Substituting these into $S_i = \sum_{i=1}^n a_{ij}x_j$ results in (\ref{e_A'}).
\end{proof}
The next result is the main theorem of this section.
\begin{theorem}\label{thm:solutions_in_C_star}
Consider a linear network where $H=1$, $m=1$, $d_x=n$, $d_y=p$ and $d_1 = d$ with generic input/output vectors and generic regularization matrices. Then, there are at most
\begin{equation*}
    \mathcal{B}_{\mathbb{C}^*} = (4p)^d
\end{equation*} solutions of (\ref{eq:gradLoss}) for which $a_{11}, \ldots, a_{dn} \in \mathbb{C}^*$.
\end{theorem}
\begin{proof}
The overall strategy for this proof is to eliminate $b_{ki}$ using (\ref{e_b}), and then substitute the expressions for the $b_{ki}$ variables so obtained into (\ref{e_a}). That will give us a polynomial system in just the $a_{ij}$ variables. We will then compute the BKK bound for this "reduced" system. That will give us an upper bound on the number of solutions in $(\mathbb{C}^*)^{d \times n}$.

We first express $b_{k2}$, \ldots, $b_{kd}$ in terms of $b_{k1}$. By the equations in (\ref{e_A'}) each $S_i$ is equal to $a_{i1}$ scaled by some constant. Since each $a_{ij}$ is non-zero by assumption, each $S_i$ is non-zero. Then (\ref{e_b}) implies
\begin{equation*}
    \frac{b_{k1} \sigma_{k1}}{S_1} = \frac{b_{k2} \sigma_{k2}}{S_2} =\cdots = \frac{b_{kd} \sigma_{kd}}{S_d}.
\end{equation*}
Hence, $b_{k2} = b_{k1}\left(\frac{\sigma_{k1}}{\sigma_{k2}} \frac{S_2}{S_1}\right)$, \ldots, $b_{kd} = b_{k1} \left(\frac{\sigma_{k1}}{\sigma_{k,d}} \frac{S_d}{S_1}\right)$. Substituting the preceding expressions for $b_{k2}$ through $b_{kd}$ into $R_k$, we get:
\begin{equation*}
\begin{split}
    R_k &= b_{k1} \left[ S_1 + \left(\frac{\sigma_{k1}}{\sigma_{k2}} \frac{S_2^2}{S_1}\right)  + \cdots + \left(\frac{\sigma_{k1}}{\sigma_{kd}} \frac{S_d^2}{S_1}\right) \right] - y_k.
\end{split}
\end{equation*}
Using this expression in (\ref{e_b}) and with a bit of algebra we obtain
\begin{equation*}
\begin{split}
    b_{ki} &= \frac{y_k S_i}{\sigma_{ki}} \frac{1}{\left(1 + T_k\right)}
\end{split}
\end{equation*}
where $T_k = \frac{S_1^2}{\sigma_{k1}} + \frac{S_2^2}{\sigma_{k2}} + \cdots +\frac{S_d^2}{\sigma_{kd}}$. Note
that the right hand side of this identity is a rational 
function of $a_{ij}$ only. We substitute these expressions for $b_{ki}$ into the first equation of (\ref{e_a}) to result in
\begin{equation*}
    \begin{split}
        -\frac{\lambda_{1j}a_{1j}}{x_j} &= -\sum_{k=1}^p \frac{y_k^2 S_1}{\sigma_{k1}(1+T_k)^2}.
    \end{split}
\end{equation*}
Now, (\ref{e_a'}) implies $-\frac{\lambda_{1j}a_{1j}}{x_j} = -\frac{\lambda_{11}a_{11}}{x_1}$ and using (\ref{e_A'}) for $S_1$, the above equation gives 
\begin{equation*}\label{e_5.1}
    \frac{1}{\left(\sum_{j=1}^n \frac{x_j^2}{\lambda_{1j}}\right)} - \left[ \frac{y_1^2}{\sigma_{11}(1+T_1)^2} + \frac{y_2^2}{\sigma_{21}(1+T_2)^2} +
      \cdots  +
        \frac{y_p^2}{\sigma_{p1}(1+T_p)^2}
        \right] = 0.
\end{equation*}
Clearing denominators we arrive at 
\begin{equation}\label{e_6.1}
    \begin{split}
        \frac{(1+T_1)^2(1+T_2)^2\cdot\cdot\cdot(1 + T_p)^2}{\left(\sum_{j=1}^n \frac{x_j^2}{\lambda_{1j}}\right)} & - \sum_{k=1}^p \frac{y_k^2}{\sigma_{k1}} \prod_{i\neq k}^p (1+T_i)^2  = 0.
    \end{split}
\end{equation}
Each $T_k$ is a polynomial in the variables $a_{11}, \ldots, a_{d1}$. Hence, (\ref{e_6.1}) is a polynomial equation in these variables only. A total of $d$ such polynomials comprises a system of $d$ polynomial equations in $d$ variables. The solution set contains the solutions
to (\ref{eq:gradLoss}) when $a_{ij} \in \CC^*$, but it also 
contains components with positive dimension, for instance, 
the codimension two algebraic sets given by $1+T_i = 1+T_j = 0$ for $i \neq j$. We proceed to introduce regularization for this system as well. To equation $i = 1, \ldots, d$ we add the regularization term $\mu_i a_{di}$ where $\mu_i$ is a regularization parameter. A standard argument using the generalized Sard's theorem and the implicit function theorem guarantees that, for almost all choices of $\mu_i$, $i=1,\ldots, d$, all solutions to the regularized system are isolated and nondegenerate, and for sufficiently small $\mu_i$, as they shrink to $0$ uniformly, the solutions to the regularized system either converge to components of the solution set of the non-regularized system (including the solutions we want to count) or diverge to infinity. For an excellent discussion of regularization in the context of solving systems of polynomials we refer to Section 3.2 in \cite{MCTH22}. 
The Newton polytope of (\ref{e_6.1}) and its regularization is $(4p)\Delta_d = \mathrm{conv}\{0, 4p \cdot e_1, ..., 4p \cdot e_d\}$ where $e_i$ is the $i$-th standard basis vector of $\RR^d$. Thus, we obtain a zero-dimensional polynomial system with $d$ equations in $d$ variables where 
the Newton polytope of each equation is $(4p)\Delta_d$.
Therefore the BKK bound for this system is given by
the normalized volume of $(4p)\Delta_d$ which is equal to $(4p)^d$. This completes the proof.
\end{proof}
\begin{remark}
Theorem \ref{thm:solutions_in_C_star} gives a bound 
on the number of critical points of the regularized loss function when the weights in $A=W_1$ are all nonzero. In the following section we will relax this requirement after treating patterns of zeros for the $H=1$ and $m=1$ case. 
We also note that in this theorem we have not made any 
assumptions about the weights in $B=W_2$.
\end{remark}
\begin{remark} It is an open question to find a reduced system for $m>1$ even when $H=1$ that will yield a better 
bound on the number of critical points. In particular, is there a systematic way to eliminate the weights in $B=W_2$?
\end{remark}

\section{Patterns of Zeros}
\label{sec4}

This section is about patterns we observe in the critical points of the regularized loss function of a deep linear neural network. All of our results are about networks trained on a single data point ($m=1$). However, we believe the phenomena we observe are more general and persist for arbitrary $m$. To support this claim we provide computational evidence. We start with the observation that if in a critical point an entry of a weight matrix is zero then 
the entire row or column of that entry must also consist of zeros. 
\begin{proposition} \label{prop:no-stray-zeros}
Suppose an arbitrary deep linear neural network is trained on
a single generic data point $x$ with the corresponding generic output vector $y$. Then, if in a critical point of $\Loss^\Lambda$ the $(i,j)$ entry of a weight matrix is zero, then either the $i$th row or the $j$th column of the same weight matrix is zero. 
\end{proposition}
\begin{proof}
We let $d_x = n$ and $d_y = p$. Suppose the $(i,j)$ entry of $Z=W_k \in \CC^{r \times s}$ is zero. 
Then 
\begin{equation*}
\frac{\partial \Loss^\Lambda}{\partial Z} = U_Z^T\left(W xx^T  - yx^T \right)V_Z^T + \Lambda  \circ Z = 0
\end{equation*}
where $U_Z = U_k$, $V_Z = V_k$ and $\Lambda=\Lambda_k$ as in
(\ref{eq:gradLoss}). It follows that 
\begin{equation}\label{eq:no-stray-zeros}
\begin{pmatrix}
         D_1E_1 & \cdots & D_1 E_s\\
          \vdots & & \vdots \\
         D_r E_1 & \cdots & D_r E_s\\
    \end{pmatrix}   =
 -\begin{pmatrix}
         \lambda_{11} z_{11} & \cdots & \lambda_{1s} z_{1s}\\
          \vdots & & \vdots \\
         \lambda_{r1} z_{r1} & \cdots & \lambda_{rs} z_{rs}\\
    \end{pmatrix}   
\end{equation}
where
\begin{equation*}
\begin{split}
        D_i &= (U_Z^T)_{i1} R_1 + \cdots + (U_Z^T)_{ip} R_p\\
        E_j &=  x_1 (V_Z^T)_{1j} + \cdots + x_n (V_Z^T)_{nj}\\
    \end{split}
\end{equation*}
and $R_k = (Wx - y)_k$. If $z_{ij} = 0$ then $D_i=0$ or $E_j=0$
and the result follows. 
\end{proof}
It is worth noting that if row $i$ of the weight matrix $W_k$ is $0$, it means that the $i$th neuron on layer $k$ can be removed, and the column $i$ of $W_{k+1}$ can be deleted. 
Conversely, if column $j$ of the weight matrix $W_k$ is $0$,
then this column can be deleted and the $j$th neuron on layer 
$k-1$ can be removed.
These statements have the following stronger counterpart. 
\begin{theorem} \label{thm:row-zero-column-zero}
Suppose an arbitrary deep linear neural network is trained on
a single generic data point $x$ with the corresponding generic output vector $y$. Then in a critical point of $\Loss^\Lambda$ the row $i$ of $W_{k-1}$ is zero if and only if  the $i$th column 
of $W_{k}$ is zero.
\end{theorem}
\begin{proof} We retain our notation from the previous proof,
and let $P = W_{k-1}$ and $Z = W_k$. Then $V_Z = P\cdots W_2W_1$, and if row $i$ of $P$ is zero, then row $i$ of $V_Z$ is zero. It follows that $E_i = x_1 (V_Z^T)_{1i} + \cdots + x_n (V_Z^T)_{ni} = 0$
and from (\ref{eq:no-stray-zeros}) we conclude that the $i$th column of $Z$ is zero. Conversely, suppose the $i$th column
of $Z$ is zero.  Since $(U_P)^T = Z^T\cdots W_H^TW_{H+1}^T$, then row $i$ of $(U_P)^T$ is zero. It follows that $D_i = (U_P^T)_{i1} R_1 + \cdots + (U_P^T)_{ip} R_p  = 0$
and the result follows again from (\ref{eq:no-stray-zeros}).
\end{proof}
For the weight matrices $W_1$ and $W_{H+1}$, Proposition \ref{prop:no-stray-zeros} has stronger corollaries.
\begin{corollary} \label{cor:matrix-1}
Suppose an arbitrary deep linear neural network is trained on
a single generic data point $x$ with the corresponding generic output vector $y$. Then, if in a critical point of $\Loss^\Lambda$ the $(i,j)$ entry of $W_1$ is zero, then the $i$th row of $W_1$
must be zero.
\end{corollary}
\begin{proof}
This follows from equations (\ref{e_a'}).
\end{proof}
\begin{corollary} \label{cor:matrix-H+1}
Suppose an arbitrary deep linear neural network  is trained on a single generic data point $x$ with the corresponding generic output vector $y$. Then, if in a critical point of $\Loss^\Lambda$ the $(i,j)$ entry of $W_{H+1}$ is zero 
 then the $j$th column  of $W_{H+1}$ must be zero.
\end{corollary}
\begin{proof}
Let $Z=W_{H+1}$ and suppose $z_{ij}=0$. Then 
$D_iE_j=0$ as in the proof of Proposition \ref{prop:no-stray-zeros}. We observe that $U_Z$ is the identity matrix
and therefore $D_i = R_i = (Wx-y)_i$. If column $j$ of $Z$
is not zero, then $E_j \neq 0$ and therefore $D_i=0$. 
Hence $R_i=0$ and this means that $(Wx)_i= y_i$. But this
also means that $i$th row of $Z$ is zero, and therefore
the $i$th row of $W$ is zero. But then $y_i=0$ which contradicts the genericity of $y$. We conclude that the 
$j$th column of $Z=W_{H+1}$ must be zero. 
\end{proof}
For a network with one hidden layer trained on a single data point we gave a bound on the number of critical points in which
the entries of the first weight matrix are in $\CC^*$; see Theorem \ref{thm:solutions_in_C_star}. Now we expand this count
to all complex critical points. 
\begin{corollary} \label{cor:reduced} Consider a linear network with $H=1$, $m=1$, $d_x=n$, $d_y=p$ and $d_1 = d$ with generic input/output vectors and generic regularization matrices. Then, there are at most
\begin{equation*}
    \mathcal{B}_{\mathbb{C}} = (1+4p)^d
\end{equation*} complex solutions of (\ref{eq:gradLoss}).
\end{corollary}
\begin{proof}
By the above results the possible 
patterns of zeros in critical points are completely 
determined by the rows of $W_1$ that are zero. Since 
the corresponding columns of $W_2$ must be zero we can 
reduce the neural network by removing the neurons matching
the zero rows of $W_1$. If $r$ rows of $W_1$ are zero, 
Theorem \ref{thm:solutions_in_C_star} implies that there
are at most $(4p)^{d-r}$ critical points where the remaining rows  of $W_1$ are in $\CC^*$. Therefore 
the number of all complex critical points for this network
is 
$$ \mathcal{B}_{\CC} = \sum_{r=0}^d {d \choose r} (4p)^{d-r}  = (1 + 4p)^d.$$
\end{proof}
\begin{example} Consider a linear neural network with one hidden layer trained on a single input vector
where $d_x=d_y=d_1=2$. The classic B\'ezout bound on the 
number of complex critical points is
$3^8 = 6561$. The BKK bound yields $1089$. 
According to Theorem \ref{thm:solutions_in_C_star} the 
number of critical points in $(\CC^*)^8$ is bounded by $\mathcal{B}_{\CC^*}= 64$, and by the above corollary the number of critical points in $\CC^8$ is bounded by $\mathcal{B}_{\CC} = 81$. These solutions are expected to come
in four flavors:
$$ \begin{array}{cccc}
       W_2 & W_1 & W_2 & W_1 \\
    \left(\begin{array}{cc} * & * \\ * & * \end{array} \right. &
    \left. \begin{array}{cc} * & * \\ * & * \end{array}  \right) &
    \left( \begin{array}{cc} 0 & * \\ 0 & * \end{array} \right. &
    \left. \begin{array}{cc} 0 & 0 \\ * & * \end{array} \right)\\
    \\
    \left(\begin{array}{cc} * & 0 \\ * & 0 \end{array} \right.&
    \left. \begin{array}{cc} * & * \\ 0 & 0 \end{array} \right) &
    \left( \begin{array}{cc} 0 & 0 \\ 0 & 0 \end{array} \right. &
    \left. \begin{array}{cc} 0 & 0 \\ 0 & 0 \end{array} \right) \\
\end{array}
$$
The first type with  full support corresponds to critical
points in $(\CC^*)^8$, and there are at most $64$ of them. 
The next two types have zeros in the first or second row of $W_1$, and there are at most $8$ such solutions for each type. Finally,
there is the last type which corresponds to the origin. Using {\em HomotopyContinuation.jl} we see that there are actually $33$ critical points in $\CC^8$ of which $16$ are in $(\CC^*)^8$. \hfill$ \triangle$
\end{example}
\begin{example} Not every possible pattern of zeros is realized. A simple example comes from the case with
$d_x=d_y=1$ and $d_1=2$ when $H=m=1$. Again one expects
four types of critical points, but there are no complex critical points with only nonzero coordinates. The total number of critical points is $9$. A more interesting example 
arises from the case when $H=d_x=d_y=d_1=2$ with $m=1$.
The total number of possible types of critical points is
ten. However, the following two types of solutions do not 
appear:
$$\begin{array}{ccc}
       W_3 & W_2 & W_1  \\
    \begin{array}{cc} 0 & * \\ 0 & * \end{array}  &
    \begin{array}{cc} 0 & 0 \\ * & * \end{array}   &
    \begin{array}{cc} * & * \\ * & * \end{array}  \\
    \\
    \begin{array}{cc} * & 0 \\ * & 0 \end{array}  &
    \begin{array}{cc} * & * \\ 0 & 0 \end{array}   &
    \begin{array}{cc} * & * \\ * & * \end{array}  
\end{array}$$
\hfill$ \triangle$
\end{example}
\begin{conjecture}
All our results in this section are for linear neural
networks trained on a single data point. However, 
in our computations we observe that Proposition \ref{prop:no-stray-zeros}, Theorem \ref{thm:row-zero-column-zero}, Corollary \ref{cor:matrix-1}, and Corollary \ref{cor:matrix-H+1} hold for $m=2$. We conjecture these results to be true
for any $m$.
\end{conjecture}

\section{Computations and Conclusions}
\label{sec5}
In this section we present results showing the number of critical points for different network architectures, i.e. different values of $H, m, d_i, d_x$ and $d_y$. When 
$H=2$ we use $d_1=d_2$ and denote it by $d_i$. The results are summarized in Tables \ref{tab:Table_H1m1} - \ref{tab:Table_H2m2} which can be found in the Appendix. Each case yields a polynomial system whose coefficients are determined by the entries of the data matrices $X \in \mathbb{R}^{d_x \times m}$ and $Y \in \mathbb{R}^{d_y \times m}$, and the regularization constants from $\Lambda_i$. The entries of $X$ and $Y$ are drawn i.i.d. from a Gaussian distribution with mean 0 and variance 1. The entries of $\Lambda_i$ are drawn i.i.d. from the uniform distribution between 0 and 1. For each case, all isolated solutions to each of the 100 samples are computed using \textit{HomotopyContinuation.jl} on a single Amazon EC2 instance. Instances used were of type \textit{c5.4xlarge} (16 vCPUs, 8 cores, 32 GB memory) or \textit{c5.9xlarge} (36 vCPUs, 18 cores, 72 GB memory), depending on the size of the polynomial system. Most systems took less than an hour of wall-clock time to solve.

Table \ref{tab:Table_H1m1} through Table \ref{tab:Table_H2m2} expand the results for $H, m \in \{1,2\}$ reported 
in \cite{MCTH22} where the authors used {\em Bertini}. We report the classical B\'ezout bound, the BKK bound, the number of complex critical points ($N_\CC$), the number of complex critical points in the algebraic torus ($N_{\CC^*}$), 
and the maximum number of real critical points observed ($\max\{N_{\mathbb{R}}\}$). Table \ref{tab:Table_H1m1} contains two additional columns reporting $\mathcal{B}_{\CC^*}$ and $\mathcal{B}_{\CC}$ of Theorem 
\ref{thm:solutions_in_C_star} and Corollary \ref{cor:reduced},
respectively. 

We observe that while the BKK bound is much better than the
B\'ezout bound, even the BKK bound is far from  the actual number of complex critical points. The number of real critical
points is even lower. The gradient equations (\ref{eq:gradLoss}) yield very structured and very sparse polynomials. Though the computed BKK bounds are not tight enough, it would be interesting to study the Newton polytopes of these equations in the hope of finding formulas for their
mixed volumes. 
Theorem \ref{thm:solutions_in_C_star}
provides much tighter bounds for the case of $H=1$ and $m=1$. We believe it is worthwhile to attempt to generalize this theorem. However, the real benefit of this theorem or its
possible generalizations is in creating custom-made homotopies
that will track only $\mathcal{B}_{\CC}$ paths. This will require constructing appropriate initial systems. 

We would like to also note that while some of our results
explain some of the numbers in the tables, there are 
other patterns waiting to be explained. Proposition \ref{prop:sameBKK} tells us that the BKK bounds 
in Table \ref{tab:Table_H1m1} and \ref{tab:Table_H1m2}, and 
those in Table \ref{tab:Table_H2m1} and \ref{tab:Table_H2m2}
for the same $d_i, d_x, d_y$ must match. From the limited
computations we have, it appears that the BKK bounds
for the identical parameters except $d_x$ and $d_y$ values
switched (for instance $d_x = 2$ and $d_y=3$ versus $d_x=3$ and $d_y=3$) coincide. We invite the reader to unearth
more patterns.

\bigskip
\medskip


\appendix
\section{Appendix} \label{sec:appendix}
The following four tables contain the results of our computational experiments. See Section \ref{sec5} for details.

\begin{table*}[h]
\begin{tabular}{ccccccccccc}
     $d_i$ & $d_x$ & $d_y$ & $N$ & CBB & BKK & $\mathcal{B}_{\mathbb{C}}$ & $\mathcal{B}_{\mathbb{C}^*}$ & $N_{\mathbb{C}}$ & $N_{\mathbb{C}^*}$ & $\max\{N_{\mathbb{R}}\}$ \\
     \hline
     1 & 1 & 1 & 2 & 9 & 5 & 5 & 4 & 5 & 4 & 3\\ 
    1 & 2 & 1 & 3 & 27 & 9 & 5 & 4 & 5 & 4 & 3\\ 
    1 & 3 & 1 & 4 & 81 & 13 & 5 & 4 & 5 & 4 & 3\\ 
    1 & 4 & 1 & 5 & 243 & 17 & 5 & 4 & 5 & 4 & 3\\ 
    1 & 1 & 2 & 3 & 27 & 9 & 9 & 8 & 9 & 8 & 3\\ 
    1 & 2 & 2 & 4 & 81 & 33 & 9 & 8 & 9 & 8 & 3\\ 
    1 & 3 & 2 & 5 & 243 & 73 & 9 & 8 & 9 & 8 & 3\\ 
    1 & 4 & 2 & 6 & 729 & 129 & 9 & 8 & 9 & 8 & 3\\ 
    1 & 1 & 3 & 4 & 81 & 13 & 13 & 12 & 13 & 12 & 3\\ 
    1 & 2 & 3 & 5 & 243 & 73 & 13 & 12 & 13 & 12 & 3\\ 
    1 & 3 & 3 & 6 & 729 & 245 & 13 & 12 & 13 & 12 & 3\\ 
    1 & 4 & 3 & 7 & 2187 & 593 & 13 & 12 & 13 & 12 & 3\\ 
    1 & 1 & 4 & 5 & 243 & 17 & 17 & 16 & 17 & 16 & 3\\ 
    1 & 2 & 4 & 6 & 729 & 129 & 17 & 16 & 17 & 16 & 3\\ 
    1 & 3 & 4 & 7 & 2187 & 593 & 17 & 16 & 17 & 16 & 3\\ 
    1 & 4 & 4 & 8 & 6561 & 1921 & 17 & 16 & 17 & 16 & 3\\ 
    2 & 1 & 1 & 4 & 81 & 25 & 25 & 16 & 9 & 0 & 5\\ 
    2 & 2 & 1 & 6 & 729 & 81 & 25 & 16 & 9 & 0 & 5\\ 
    2 & 3 & 1 & 8 & 6561 & 169 & 25 & 16 & 9 & 0 & 5\\ 
    2 & 4 & 1 & 10 & 59049 & 289 & 25 & 16 & 9 & 0 & 5\\ 
    2 & 1 & 2 & 6 & 729 & 81 & 81 & 64 & 33 & 16 & 9\\ 
    2 & 2 & 2 & 8 & 6561 & 1089 & 81 & 64 & 33 & 16 & 9\\ 
    2 & 3 & 2 & 10 & 59049 & 5329 & 81 & 64 & 33 & 16 & 9\\ 
    2 & 4 & 2 & 12 & 531441 & 16641 & 81 & 64 & 33 & 16 & 9\\ 
    2 & 1 & 3 & 8 & 6561 & 169 & 169 & 144 & 73 & 48 & 9\\ 
    2 & 2 & 3 & 10 & 59049 & 5329 & 169 & 144 & 73 & 48 & 9\\ 
    2 & 3 & 3 & 12 & 531441 & 60025 & 169 & 144 & 73 & 48 & 9\\ 
    2 & 4 & 3 & 14 & 4782969 & 351649 & 169 & 144 & 73 & 48 & 9\\ 
    2 & 1 & 4 & 10 & 59049 & 289 & 289 & 256 & 129 & 96 & 9\\ 
    2 & 2 & 4 & 12 & 531441 & 16641 & 289 & 256 & 129 & 96 & 9\\ 
    2 & 3 & 4 & 14 & 4782969 & 351649 & 289 & 256 & 129 & 96 & 9\\ 
    3 & 1 & 1 & 6 & 729 & 125 & 125 & 64 & 13 & 0 & 7\\ 
    3 & 2 & 1 & 9 & 19683 & 729 & 125 & 64 & 13 & 0 & 7\\ 
    3 & 3 & 1 & 12 & 531441 & 2197 & 125 & 64 & 13 & 0 & 7\\ 
    3 & 4 & 1 & 15 & 14348907 & 4913 & 125 & 64 & 13 & 0 & 7\\ 
    3 & 1 & 2 & 9 & 19683 & 729 & 729 & 512 & 73 & 0 & 15\\ 
    3 & 2 & 2 & 12 & 531441 & 35937 & 729 & 512 & 73 & 0 & 15\\ 
    3 & 3 & 2 & 15 & 14348907 & 389017 & 729 & 512 & 73 & 0 & 15\\ 
    3 & 4 & 2 & 18 & 387420489 & 2146689 & 729 & 512 & 73 & 0 & 15\\ 
    3 & 1 & 3 & 12 & 531441 & 2197 & 2197 & 1728 & 245 & 64 & 19\\ 
    3 & 2 & 3 & 15 & 14348907 & 389017 & 2197 & 1728 & 245 & 64 & 23\\
    \hline
\end{tabular}
\caption{\begin{tiny} Case: $H=1, m=1$. 
   $d_i$ = \# neurons in each layer, $d_x$ = input dim. $d_y$ = output dim. $N$ = \# weights. CBB and BKK classical B\'ezout and BKK bound. $\mathcal{B}_{\mathbb{C}}$ and $\mathcal{B}_{\mathbb{C}^*}$: bounds from
   Thm \ref{thm:solutions_in_C_star} and Cor \ref{cor:reduced} in $(\mathbb{C})^N$ and  $(\mathbb{C}^*)^N$. $N_{\mathbb{C}}$ and $N_{\mathbb{C}^*}$ = actual \# critical points in $(\mathbb{C})^N$ and  $(\mathbb{C}^*)^N$, respectively. $\max\{N_{\mathbb{R}}\}$ = max \#  real sols. observed for each case. \end{tiny}}
\label{tab:Table_H1m1}
\end{table*}

\begin{table*}[h]
\begin{tabular}{ccccccccc}
    $d_i$ & $d_x$ & $d_y$ & $N$ & CBB & BKK & $N_{\mathbb{C}}$ & $N_{\mathbb{C}^*}$ & $\max\{N_\mathbb{R}\}$\\ 
    \hline
    1 & 1 & 1 & 2 & 9 & 5 & 5 & 4 & 3\\ 
    1 & 2 & 1 & 3 & 27 & 9 & 9 & 8 & 3\\ 
    1 & 3 & 1 & 4 & 81 & 13 & 9 & 8 & 3\\ 
    1 & 4 & 1 & 5 & 243 & 17 & 9 & 8 & 3\\ 
    1 & 1 & 2 & 3 & 27 & 9 & 9 & 8 & 3\\ 
    1 & 2 & 2 & 4 & 81 & 33 & 17 & 16 & 5\\ 
    1 & 3 & 2 & 5 & 243 & 73 & 17 & 16 & 9\\ 
    1 & 4 & 2 & 6 & 729 & 129 & 17 & 16 & 5\\ 
    1 & 1 & 3 & 4 & 81 & 13 & 13 & 12 & 3\\ 
    1 & 2 & 3 & 5 & 243 & 73 & 25 & 24 & 9\\ 
    1 & 3 & 3 & 6 & 729 & 245 & 25 & 24 & 9\\ 
    1 & 4 & 3 & 7 & 2187 & 593 & 25 & 24 & 5\\ 
    1 & 1 & 4 & 5 & 243 & 17 & 17 & 16 & 3\\ 
    1 & 2 & 4 & 6 & 729 & 129 & 33 & 32 & 9\\ 
    1 & 3 & 4 & 7 & 2187 & 593 & 33 & 32 & 9\\ 
    1 & 4 & 4 & 8 & 6561 & 1921 & 33 & 32 & 9\\ 
    2 & 1 & 1 & 4 & 81 & 25 & 9 & 0 & 5\\ 
    2 & 2 & 1 & 6 & 729 & 81 & 33 & 16 & 9\\ 
    2 & 3 & 1 & 8 & 6561 & 169 & 33 & 16 & 9\\ 
    2 & 4 & 1 & 10 & 59049 & 289 & 33 & 16 & 9\\ 
    2 & 1 & 2 & 6 & 729 & 81 & 33 & 16 & 9\\ 
    2 & 2 & 2 & 8 & 6561 & 1089 & 225 & 192 & 25\\ 
    2 & 3 & 2 & 10 & 59049 & 5329 & 225 & 192 & 25\\ 
    2 & 4 & 2 & 12 & 531441 & 16641 & 225 & 192 & 29\\ 
    2 & 1 & 3 & 8 & 6561 & 169 & 73 & 48 & 9\\ 
    2 & 2 & 3 & 10 & 59049 & 5329 & 705 & 656 & 29\\ 
    2 & 3 & 3 & 12 & 531441 & 60025 & 705 & 656 & 29\\ 
    2 & 4 & 3 & 14 & 4782969 & 351649 & 705 & 656 & 29\\ 
    2 & 1 & 4 & 10 & 59049 & 289 & 129 & 96 & 9\\ 
    2 & 2 & 4 & 12 & 531441 & 16641 & 1601 & 1536 & 33\\ 
    2 & 3 & 4 & 14 & 4782969 & 351649 & 1601 & 1536 & 33\\ 
    3 & 1 & 1 & 6 & 729 & 125 & 13 & 0 & 7\\ 
    3 & 2 & 1 & 9 & 19683 & 729 & 73 & 0 & 19\\ 
    3 & 3 & 1 & 12 & 531441 & 2197 & 73 & 0 & 15\\ 
    3 & 4 & 1 & 15 & 14348907 & 4913 & 73 & 0 & 19\\ 
    3 & 1 & 2 & 9 & 19683 & 729 & 73 & 0 & 18\\ 
    3 & 2 & 2 & 12 & 531441 & 35937 & 1265 & 640 & 57\\ 
    3 & 3 & 2 & 15 & 14348907 & 389017 & 1265 & 640 & 65\\ 
    3 & 4 & 2 & 18 & 387420489 & 2146689 & 1265 & 640 & 61\\ 
    3 & 1 & 3 & 12 & 531441 & 2197 & 245 & 64 & 23\\ 
    3 & 2 & 3 & 15 & 14348907 & 389017 & 8825 & 6784 & 101\\
    \hline
\end{tabular}
\caption{Case: $H=1, m=2$}
\label{tab:Table_H1m2}
\end{table*}

\begin{table*}[h]
\begin{tabular}{ccccccccc}
    $d_i$ & $d_x$ & $d_y$ & $N$ & CBB & BKK & $N_{\mathbb{C}}$ & $N_{\mathbb{C}^*}$ & $\max\{N_\mathbb{R}\}$\\ 
    \hline
    1 & 1 & 1 & 3 & 125 & 17 & 17 & 16 & 9\\ 
    1 & 2 & 1 & 4 & 625 & 33 & 17 & 16 & 9\\ 
    1 & 3 & 1 & 5 & 3125 & 49 & 17 & 16 & 9\\ 
    1 & 4 & 1 & 6 & 15625 & 65 & 17 & 16 & 9\\ 
    1 & 1 & 2 & 4 & 625 & 33 & 33 & 32 & 9\\ 
    1 & 2 & 2 & 5 & 3125 & 129 & 33 & 32 & 9\\ 
    1 & 3 & 2 & 6 & 15625 & 289 & 33 & 32 & 9\\ 
    1 & 4 & 2 & 7 & 78125 & 513 & 33 & 32 & 9\\ 
    2 & 1 & 1 & 8 & 390625 & 1857 & 129 & 64 & 64\\ 
    2 & 2 & 1 & 10 & 9765625 & 23937 & 129 & 64 & 64\\ 
    2 & 3 & 1 & 12 & 244140625 & 109249 & 129 & 64 & 64\\ 
    2 & 4 & 1 & 14 & 6103515625 & 325377 & 129 & 64 & 64\\ 
    2 & 1 & 2 & 10 & 9765625 & 23937 & 641 & 384 & 64\\ 
    2 & 2 & 2 & 12 & 244140625 & 780801 & 641 & 384 & 65\\ 
    3 & 1 & 1 & 15 & 30517578125 & 667921 & 977 & 256 & 232\\
    \hline
\end{tabular}
\caption{Case: $H=2, m=1$}
\label{tab:Table_H2m1}
\end{table*}

\begin{table*}[h]
\begin{tabular}{ccccccccc}
    $d_i$ & $d_x$ & $d_y$ & $N$ & CBB & BKK & $N_{\mathbb{C}}$ & $N_{\mathbb{C}^*}$ & $\max\{N_\mathbb{R}\}$\\ 
    \hline
    1 & 1 & 1 & 3 & 125 & 17 & 17 & 16 & 9\\ 
    1 & 2 & 1 & 4 & 625 & 33 & 33 & 32 & 9\\ 
    1 & 3 & 1 & 5 & 3125 & 49 & 33 & 32 & 9\\ 
    1 & 4 & 1 & 6 & 15625 & 65 & 33 & 32 & 9\\ 
    1 & 1 & 2 & 4 & 625 & 33 & 33 & 32 & 9\\ 
    1 & 2 & 2 & 5 & 3125 & 129 & 81 & 80 & 24\\ 
    1 & 3 & 2 & 6 & 15625 & 289 & 81 & 80 & 17\\ 
    1 & 4 & 2 & 7 & 78125 & 513 & 81 & 80 & 17\\
    2 & 1 & 1 & 8 & 390625 & 1857 & 129 & 64 & 64\\ 
    2 & 2 & 1 & 10 & 9765625 & 23937 & 641 & 384 & 65\\ 
    2 & 3 & 1 & 12 & 244140625 & 109249 & 641 & 384 & 64\\ 
    2 & 4 & 1 & 14 & 6103515625 & 325377 & 641 & 384 & 72\\ 
    2 & 1 & 2 & 10 & 9765625 & 23937 & 641 & 384 & 72\\
    \hline
\end{tabular}
\caption{Case: $H=2, m=2$}
\label{tab:Table_H2m2}
\end{table*}

\end{document}